\documentclass[12pt]{amsart}
\nonstopmode
\RequirePackage[colorlinks,citecolor=blue,urlcolor=blue,linkcolor=blue]{hyperref}
\hypersetup{
colorlinks = true,
citecolor=blue,
urlcolor=blue,
linkcolor=blue,
pdfauthor = {Alexey Kuznetsov},
pdfkeywords = {Riemann zeta function, series expansion, Mellin transform, incomplete Gamma function, generalized hypergeometric function, Meijer G-function},
pdftitle = {Series expansions for the Riemann zeta function},
pdfpagemode = UseNone
}
\usepackage{graphicx,xspace,colortbl}
\usepackage{amsmath,amsthm,amsfonts}
\usepackage{color}
\usepackage{fancybox}
\usepackage{epsfig}
\usepackage{subfig}
\usepackage{pdfsync}
    
    \def\qed{\hfill$\sqcap\kern-8.0pt\hbox{$\sqcup$}$\\}
    \def\beq{\begin{eqnarray}}
    \def\eeq{\end{eqnarray}}
    \def\beqq{\begin{eqnarray*}}
    \def\eeqq{\end{eqnarray*}}

    \def\re{\textnormal {Re}}
    \def\im{\textnormal {Im}}

    \def\r{{\mathbb R}}
    \def\c{{\mathbb C}}
    	
    \def\d{{\textnormal d}}
    \def\i{{\textnormal i}}

\newtheorem{theorem}{Theorem}

\newtheorem{proposition}{Proposition}
\newtheorem{corollary}{Corollary}
\theoremstyle{definition}

\title{Series expansions for the Riemann zeta function}

\author{Alexey Kuznetsov}
\address{Department of Mathematics and Statistics, York University,
Toronto, Ontario, M3J1P3,
Canada}
\email{akuznets@yorku.ca}

\date{\today}
\begin{document}
\maketitle

 \begin{center}
  \emph{To the memory of Richard Paris}
\end{center}

\begin{abstract}
We prove a general result  on representing the Riemann zeta function as a convergent infinite series in a complex vertical  strip containing the critical line. We use this result to re-derive known expansions as well as to discover new series representations of the Riemann zeta function in terms of the incomplete gamma functions, generalized hypergeometric functions and Meijer $G$-functions.  
\end{abstract}

\bigskip

 \noindent {\it Keywords}: Riemann zeta function, approximate formula, Mellin transform, incomplete gamma function, generalized hypergeometric function, Meijer $G$-function
{\vskip 0.15cm}
 \noindent {\it 2020 Mathematics Subject Classification }:  11M06, 41A58
\bigskip


\section{Introduction}

The asymptotics of the Riemann zeta function on the critical line was a topic that greatly interested Richard Paris. He wrote the first paper \cite{Paris1994} on this subject in 1994. This was followed by a series of papers   \cite{Paris1994b,Paris2000,Paris_Cang1997,Paris_Cang1997b} in the 1990s, two of them co-authored with his PhD student Shuang Cang. Richard returned to this topic in 2009 \cite{Paris2009} and, very recently, in 2022 \cite{Paris2022}. 
 To explain the origin of the motivation for his work on the Riemann zeta function, we need to start with the Riemann-Siegel formula
(see  \cite{Titchmarsh1987}[Section 4.17]):
\begin{align}\label{RS_formula}
Z(t):=e^{\i \theta(t)} \zeta\big(\tfrac{1}{2}+\i t\big)&=2 \sum\limits_{n=1}^{N_t} \frac{\cos(\theta(t)-t \ln(n))}{\sqrt{n}} \\ \nonumber
&+ (-1)^{N_t-1} \Big( \frac{2\pi}{t} \Big)^{\frac{1}{4}} 
\frac{\cos(\pi \omega^2/2+3\pi /8)}{\cos(\pi \omega)} + O(t^{-3/4}), \;\;\;
t\to +\infty.
\end{align}
Here $\zeta(s)$ is the Riemann zeta function, $\theta$ is defined as
$$
\theta(t):=\Big[ \frac{\Gamma(1/4+\i t/2)}{
\Gamma(1/4-\i t/2)} \Big]^{1/2} \pi^{- \i t/2},
$$
and  $N_t:=\lfloor \sqrt{t/(2\pi)} \rfloor$ 
and $\omega:=1+2(N_t-\sqrt{t/(2\pi)})$. The function $Z(t)$ is even, it is real for real values of $t$ and satisfies $|Z(t)|=|\zeta(1/2+\i t)|$, which makes it a convenient tool for detecting zeros of $\zeta(s)$ on the critical line $\re(s)=1/2$. The sum in \eqref{RS_formula} is the leading term in the asymptotic expansion of $Z(t)$ (it is often called ``the main sum"). We also included in \eqref{RS_formula} the first order correction term, which  has asymptotic order $O(t^{-1/4})$. The higher order correction terms can also be given explicitly, see \cite{Titchmarsh1987}[Theorem 4.16]. In applications we are usually interested to compute $Z(t)$ for $t$ very large, where most computational effort is spent in  evaluating the main sum in \eqref{RS_formula}. See \cite{Hiary2018} for results of such computations for $t$ as large as $10^{34}$. For such extremely large values of $t$ the error term $O(t^{-3/4})$ is already very small. Gabcke in his PhD Thesis \cite{Gabcke1979} proved that this error term is in fact less than $0.127 t^{-3/4}$ for all $t\ge 200$ and he also provided explicit bounds for higher order error terms.  

An important feature of the Riemann-Siegel formula \eqref{RS_formula} is that the main sum, which is the leading asymptotic term, is a discontinuous function of $t$. The discontinuity in the main sum in \eqref{RS_formula} carries over to a discontinuity in the asymptotic correction term. One would hope that finding asymptotic approximations to $Z(t)$ where the leading term is a smooth function of $t$ will result in smaller correction terms over wider range of values of $t$.   Berry and Keating  \cite{Berry_Keating1992} in 1992 developed a ``smoothed" version of the Riemann-Siegel formula. They derived an asymptotic approximation to $Z(t)$, where the leading term is given by a  {\it convergent} series
\begin{equation}\label{BK_formula}
Z_0(t,K)=2 \re \sum\limits_{n\ge 1} \frac{e^{\i (\theta(t)- t\ln(n))}}{\sqrt{n}}
\times \frac{1}{2} {\textnormal{erfc}}\Big( \frac{\xi(n,t)}{q(K,t)} \sqrt{\frac{t}{2}} \Big). 
\end{equation}
In the above formula $\xi(n,t):=\ln(n)-\theta'(t)$, $q^2(K,t):=K^2-\i t \theta''(t)$ and ${\textnormal{erfc}}(x)$ 
is the complementary error function. The parameter $K> 0$ in \eqref{BK_formula}  can be chosen freely. For large $t$ we have $\xi(n,t) \approx \ln (n/N_t)$ and $q^2(K,t) \approx K^2 - \i /2$, thus $Z_0(t,K)$ resembles the main sum in \eqref{RS_formula} but with the complementary error function smoothing the sharp cut-off at $n=N_t$. As Berry and Keating show in \cite{Berry_Keating1992}, in the infinite series in \eqref{BK_formula} only the terms with $n\le N_t+A K$  are important (the terms with larger $n$ are much smaller). This clarifies the meaning of the parameter $K$: it is proportional to the number of terms by which the truncation in the Riemann-Siegel formula has been smoothed. Berry and Keating also show that
the leading approximation term $Z_0(t,K)$  contains the first correction term shown in  \eqref{RS_formula} and that numerically $Z_0(t,K)$ gives a better approximation to $Z(t)$ for large $t$ when compared with the main sum in \eqref{RS_formula}.


Following the paper by Berry and Keating, Richard Paris (together with his PhD student Shuang Cang) developed several other asymptotic approximations to $Z(t)$ and investigated their properties. In \cite{Paris1994} an asymptotic formula for $Z(t)$ was derived by applying the Poisson summation to the tail of the Dirichlet series of $\zeta(s)$ for $\re(s)>1$.  In \cite{Paris_Cang1997} Paris and Cang developed another approximation to $Z(t)$, starting with the following formula 
\begin{equation}\label{PC_formula1}
Z(t)=2 \re \; e^{\i \theta(t)} \Big[ \sum\limits_{n\ge 1} n^{-s} Q(s/2,\pi \i n^2) - \frac{\pi^{s/2} e^{\pi \i s/4}}{s \Gamma(s/2)} \Big], \;\;\; s=1/2+\i t,
\end{equation}
where $Q(s,x):=\Gamma(s,x)/\Gamma(s)$ is the normalized incomplete gamma function.
Using the uniform asymptotic results for $Q(s,x)$, they found that this new approximation was more accurate than the Riemann-Siegel formula, and that it required little additional computational effort. 
Two generalizations of this approximation were studied in  \cite{Paris1994b,Paris2009,Paris_Cang1997b}.

Our goal in this paper is to state and prove a general result (Theorem \ref{thm_main} below), which allows one to derive  convergent series representations for the Riemann zeta function in a vertical strip containing the critical line. This result and its proof were inspired by the method used by Berry and Keating in \cite{Berry_Keating1992} (though very similar ideas can be found in \cite{Lavrik1968} and in 
\cite{Rubinstein_2005}[Theorem 1]). Using this general result we re-derive formula \eqref{PC_formula1} and obtain several new series representations for the Riemann zeta function, given in in terms of the incomplete Gamma functions,  generalized hypergeometric functions and Meijer $G$-functions.

\section{Results}
\label{section_results}

For $a \in \r$ and $b<\pi/2$ we denote 
\begin{equation}\label{domain_analyticity}
{\mathcal D}_{a,b}:=\Big\{(s,x) \in {\mathbb C}^2 \; : \; \re(s)>-a,\; x\neq 0, \;  |{\textnormal{arg}}(x)|<\frac{\pi}{2}-b \Big\}. 
\end{equation} 
The following theorem is our main result. 

\begin{theorem}\label{thm_main}
Let $g(z)$ be an odd function that is analytic everywhere in the vertical strip $|\re(z)|<a$ for some $a>1/2$, except for a simple pole at $z=0$ with residue equal to one. We also assume that for some  $C>0$
and $b < \frac{\pi}{2}$ we have
\begin{equation}\label{g_asymptotics}
|g(z)|< C e^{b|\im(z)|/2} \;\; {\textnormal{ for all $z$ with }} \; |\re(z)|<a \; {\textnormal{ and }} \; |\im(z)|>1.
\end{equation}
Then the following statements are true: 
\begin{itemize}
\item[(i)] There exists a function $h(s,x)$,  analytic in the domain 
${\mathcal D}_{a,b}$
 and having Mellin transform 
\begin{equation} \label{h_Melin_transform}
\int_0^{\infty}  h(s,x) x^{w-1}\d x= \Gamma(w) g(2w-s), 
\end{equation}
where $\frac{1}{2}\max(0,\re(s))<\re(w)<
\frac{1}{2}(a+\re(s))$. 
\item[(ii)] For any $\tau \in {\mathbb C}\setminus \{0\}$ with $|{\textnormal{arg}}(\tau)|<\frac{\pi}{2}-b$ and $s$ in the vertical strip $1-a<\re(s)<a$  we have 
\begin{align}\label{eqn_main}
\pi^{-s/2} \Gamma(s/2) \zeta(s)&=
\tau^{s/2} \Big[ - g(s)+ 2\sum\limits_{n\ge 1}   h(s,  \pi \tau n^2) \Big ] 
\\ \nonumber
&+  \tau^{(s-1)/2}\Big[ - g(1-s)+2\sum\limits_{n\ge 1}   h(1-s,  \pi \tau^{-1} n^2) \Big]. 
\end{align}
Both infinite series in \eqref{eqn_main} congerve absolutely and uniformly in the strip $1-a<\re(s)<a$.
\end{itemize}
\end{theorem}
\begin{proof}
For $s$ with $\re(s)>-a$ we denote by $I_s$ the interval $(\frac{1}{2}\max(0,\re(s)),
\frac{1}{2}(a+\re(s)) \subset \r$. For $(s,x) \in {\mathcal D}_{a,b}$ and $\gamma \in I_s$  we define 
\begin{equation}\label{def_h_s_x}
h_{\gamma}(s,x):= \frac{1}{2\pi \i} \int_{\gamma-\i \infty}^{\gamma+\i \infty} \Gamma(w) g(2w-s) x^{-w} \d w.
\end{equation}
First, let us prove that the above integral converges, so that $h(s,x)$ is indeed well defined. The condition  $\gamma \in I_s$ implies $\gamma>0$ and $2\gamma - \re(s)\in (0,a)$. This fact combined with our assumption that $g(z)$ is analytic in the strip $0<\re(z)<a$ show that the integrand $w\mapsto \Gamma(w) g(2w-s) x^{-w}$ has no singularities on the line of integration $w\in \gamma+\i \r$.  
From formula 8.328.1 in \cite{Jeffrey2007} we have the following asymptotic result for the gamma function:  for $z=x+\i y$ 
\begin{equation}\label{gamma_asymptotics}
\Gamma(z) \sim \sqrt{2\pi} |y|^{x-1/2} e^{-\pi |y|/2}, \;\;\; y \to \infty.  
\end{equation}
 This result and our assumption \eqref{g_asymptotics} imply that there exists a constant $C_1=C_1(\gamma,s)>0$ such that for $w=\gamma+\i y$ and $y\in \r$ we have
$$
|\Gamma(w) g(2w-s)|< C_1 e^{-(\pi/2-b) |y|} (1+|y|)^{\gamma-1/2}, \;\;\; y\to \infty.
$$
From the above asymptotic result  and the identity 
$$
|x^{-w}|=|x|^{-\gamma} e^{\arg(x) y}, \;\;\; w=\gamma+ \i y, 
$$
we conclude that the integrand in \eqref{def_h_s_x} is bounded by 
\begin{equation}\label{integrand_asymptotics}
|\Gamma(w) g(2w-s) x^{-w} | \le C_1 |x|^{-\gamma} e^{-(\pi/2-b) |y|+\arg(x) y} (1+|y|)^{\gamma-1/2},
\end{equation}
and this (as a function of $y\in \r$) is integrable as long as $(s,x) \in {\mathcal D}_{a,b}$. 
Thus $h_{\gamma}(s,x)$ is well-defined for $(s,x) \in {\mathcal D}_{a,b}$ and $\gamma \in I_s$. The estimate \eqref{integrand_asymptotics} also implies the following result, which we will need later: for any small $\epsilon>0$, $s$ in the half-plane $\re(s)>-a$ and $\gamma \in I_s$, there exists $C_2=C_2(\epsilon,s,\gamma)$ such that  $|h_{\gamma}(s,x)|<C_2 |x|^{-\gamma}$ for all $x$ with  $|\arg(x)|<\pi/2-b-\epsilon$.

Next, let us fix  $(s,x) \in {\mathcal D}_{a,b}$. The estimate \eqref{integrand_asymptotics} shows that the integrand decays exponentially fast as $\im(w)\to \infty$, uniformly in the vertical strip $\re(w) \in I_s$.  Thus, by shifting the contour of integration in this strip, we see that the function $\gamma \mapsto h_{\gamma}(s,x)$ is constant in the interval $I_s$ and the functions $\{h_{\gamma}(s,x)\}_{\gamma \in I_s}$ are analytic continuations of a single function $h(s,x)$. For any fixed $\gamma>0$, the formula \eqref{def_h_s_x} defines this function $h(s,x)$ in a domain 
$$
\Big\{(s,x) \in {\mathbb C}^2 \; : \; 2\gamma-a<\re(s)<2\gamma,\; x\neq 0, \;  |{\textnormal{arg}}(x)|<\frac{\pi}{2}-b \Big\} \subset {\mathcal D}_{a,b}.
$$
This is true since the condition $\gamma \in I_s$ is equivalent to $2\gamma-a<\re(s)<2\gamma$. As $\gamma$ ranges through all positive numbers, the union of these above smaller domains gives us the entire domain ${\mathcal D}_{a,b}$, therefore the function $h(s,x)$ is analytic in $ {\mathcal D}_{a,b}$.

As we have established above, for any fixed $s$ with $\re(s)>-a$ and any $\gamma\in I_s$, there exists $C_2=C_2(s,\gamma)$ such that $|h(s,x)|< C_2 |x|^{-\gamma}$ for all $x \in (0,\infty)$. This implies that the integral in the left-hand side of 
\eqref{h_Melin_transform} converges for all $w$ in the vertical strip 
$\frac{1}{2}\max(0,\re(s))<\re(w)<
\frac{1}{2}(a+\re(s))$, and now identity \eqref{h_Melin_transform} follows by Mellin inversion formula. This ends the proof of part (i).

Next, we denote $G(s)=\pi^{-s/2} \Gamma(s/2) \zeta(s)$. It is well known \cite{Titchmarsh1987}
that $G$ is analytic in ${\mathbb C}$, except for two simple poles at $s=0$ and $s=1$ with the corresponding residues $-1$ and $1$ and that $G$ satisfies the functional equation $G(s)=G(1-s)$. It is also known that as $\im(s)\to \infty$ in any vertical strip $|\re(s-1/2)|<A$ the function $|\zeta(s)|$ is bounded by $|s|^B$ for some $B=B(A)>0$ (in fact, one can take $B=1/6+A$, see \cite{Titchmarsh1987}[Chapter V]). Combining this estimate with \eqref{gamma_asymptotics} we conclude that for any $A>0$ there exists $B_2=B_2(A)>0$ such that 
\begin{equation}\label{eq:f_asymptotics}
|G(s)|=O\Big(|s|^{B_2} e^{-\frac{\pi}{4} |\im(s)|}\Big), \;\;\; \im(s)\to \infty,
\end{equation}
uniformly in the strip $|\re(s)-1/2|<A$. 

Let us fix $s \notin \{0,1\}$ and $\tau \in \c \setminus \{0\}$, satisfying $1-a<\re(s)<a$ and  $|\arg(\tau)|<\pi/2-b$, and take any $u>|\im(s)|$ and $c>0$ such that 
$\max(\re(s), 1-\re(s))<c<a$. We define a contour, traversed counter-clockwise 
$$
R_{c,u}:=(c-\i u, c + \i u] \cup (c+\i u, -c + \i u] \cup (-c+\i u,-c-\i u] \cup 
(-c - \i u, c-\i u]. 
$$
It is clear that $R_{c,u}$ is the boundary of a rectangle built on vertices $\pm c \pm \i u$.
Now we consider the following integral
\begin{equation}\label{def:F1}
F(s):=\frac{1}{2\pi \i} \int_{R_{c,u}} G(s+z) g(z) \tau^{-\frac{z}{2}} \d z. 
\end{equation}
By construction, the contour $R_{c,u}$ contains the three points $0$, $-s$ and $1-s$, which are the simple poles of the integrand. Applying Cauchy's Theorem we obtain 
\begin{equation}\label{eq:F1}
F(s)=G(s)-\tau^{s/2} g(-s)+\tau^{(s-1)/2} g(1-s). 
\end{equation} 
Using our assumptions on $\arg(\tau)$ and $g(z)$ (see \eqref{g_asymptotics}) and upper bound \eqref{eq:f_asymptotics} we conclude that the function 
$$
z \mapsto G(s+z) g(z) \tau^{-\frac{z}{2}}
$$
converges to zero exponentially fast as $\im(z) \to \pm \infty$, uniformly in the strip $\re(z)<a$. Thus we can take the limit in \eqref{def:F1} and \eqref{eq:F1} as $u\to +\infty$, and the integrals over the upper and lower sides of the rectangle $R_{c,u}$ will converge to zero, giving us the following result: 
\begin{align}
\label{eq:f3}
G(s)&=-\tau^{s/2}g(s)-\tau^{(1-s)/2}g(1-s)\\\nonumber
&+\frac{1}{2\pi \i} \int_{c-\i \infty}^{c+\i \infty} G(s+z) g(z) \tau^{-z/2} \d z+
\frac{1}{2\pi \i} \int_{-c+\i \infty}^{-c-\i \infty}  G(s+z) g(z) \tau^{-z/2}\d z. 
\end{align}
When deriving the above formula we also used the fact that $g(-s)=-g(s)$. 
Since $G(s+z)=G(1-s-z)$, by changing the variable of integration $z \mapsto -z$ in the second integral in the right-hand side of \eqref{eq:f3} we obtain
\begin{equation}\label{eq:f4}
G(s)=-\tau^{s/2}g(s)-\tau^{(1-s)/2}g(1-s)+H(s, \tau)+H(1-s, \tau^{-1}),
\end{equation}
where we denoted 
$$
H(s,\tau):=\frac{1}{2\pi \i} \int_{c-\i \infty}^{c+\i \infty} G(s+z) g(z) \tau^{-z/2} \d z. 
$$
The condition $c>1-\re(s)$ (that we imposed above) allows us to expand $\zeta(s+z)=\sum\limits_{n\ge 1} n^{-s-z}$ as an absolutely convergent series (this is valid for all $z$ on the vertical line  $\re(z)=c$). Applying Fubini's theorem
we have  
\begin{align}\label{eq:G1}
H(s,\tau)&=
\sum\limits_{n\ge 1} \frac{1}{2\pi \i} \int_{c-\i \infty}^{c+\i \infty} \pi^{-(s+z)/2}
\Gamma((s+z)/2)  g(z) n^{-s-z} \tau^{-z/2} \d z\\
\nonumber
&= 
2 \tau^{s/2} \sum\limits_{n\ge 1} 
\frac{1}{2\pi \i} \int_{c_1-\i \infty}^{c_1+\i \infty} \Gamma(w) g(2w-s) 
(\pi \tau n^2)^{-w}  \d w=
2 \tau^{s/2} \sum\limits_{n\ge 1} h(s,\pi \tau n^2). 
\end{align}
In the second step in the above computation we changed the variable of integration $z\mapsto w=(s+z)/2$ and denoted $c_1:=(c+\re(s))/2$. Note that 
$\max(\re(s),1/2)<c_1<(a+\re(s))/2$, which implies $c_1 \in I_s$ and justifies the last step in \eqref{eq:G1}. Formulas \eqref{eq:f4} and \eqref{eq:G1} give us the desired result \eqref{eqn_main}.  

\end{proof}

As we show next, under stronger assumption on the function $g(z)$, the function $h(s,x)$ can be given as an integral of the incomplete gamma function. 

\begin{proposition}\label{prop1}
Assume that the function $g$ satisfies the assumptions of Theorem \ref{thm_main} with some $b<0$. Then there exists a function $f : (0,\infty) \mapsto \c$,  defined via the Mellin transform 
\begin{equation}\label{def_f0}
\int_0^{\infty} f(x) x^{w-1} \d x= w g(2w), \;\;\; -a/2<\re(w)<a/2, 
\end{equation}  
which is analytic in the sector $|\arg(x)|<|b|$ and satisfies an identity $f(x)=f(1/x)$ in this sector. We have an integral representation
\begin{equation}\label{h_s_x_formula}
h(s,x)=x^{-s/2} \int_0^{\infty}  \Gamma(s/2,v) f(x/v)v^{-1 }\d v, 
\;\;\; \re(s)>-a,  \; |\arg(x)|<|b|. 
\end{equation}
\end{proposition}
\begin{proof}
For $x>0$ we define 
\begin{equation}\label{def_f}
f(x)=\frac{1}{2\pi \i} \int_{\gamma-\i \infty}^{\gamma + \i \infty} 
w g(2w) x^{-w} \d w,
\end{equation}
where $|\gamma|<a/2$. Since we assumed that $b<0$, formula \eqref{g_asymptotics} tells us that 
$$
|g(2w)|=O\big( e^{-|b| \times |\im(w)|} \big)
$$
 as $ w \to \infty$ in the strip $|\re(w)|<a/2$. Thus the integral in \eqref{def_f} converges and $f(x)$ is well defined for $x>0$. Due to the exponential decay of $g(2w)$, the function $f(x)$ is analytic in the sector $|\arg(x)|<|b|$. The fact that the function $w g(2w)$ is even implies that $f(x)=f(1/x)$ for all $x$ in the sector $|\arg(x)|<|b|$, and the fact that $wg(2w)$ is analytic in the strip $|\re(w)|<a/2$ implies that for every $\epsilon>0$ 
 \begin{equation}
 f(x)=
 \begin{cases}
 O(x^{a/2-\epsilon}), \;\;\; & x\to 0, \\
 O(x^{-a/2+\epsilon}), \;\;\; & x\to \infty,  
 \end{cases}
\end{equation}  
and this asymptotics holds in the sector $|\arg(x)|<|b|$. 

Now, according to \eqref{h_Melin_transform}, the function $x\mapsto h(s,x)$ has Mellin transform $\Phi_1(w) \Phi_2(w)$, where we denoted
$$
\Phi_1(w):=\frac{\Gamma(w)}{w-s/2}, \;\;\; \Phi_2(w):=(w-s/2) g(2w-s). 
$$
Formula (6.455.1) from \cite{Jeffrey2007} implies
\begin{equation}\label{incomplete_Gamma_Mellin}
\int_0^{\infty} x^{-\nu}\Gamma(\nu,x) x^{z-1} \d x=\frac{\Gamma(z)}{z-\nu}, \;\;\; \re(z)>0, \re(z-\nu)>0. 
\end{equation}
Therefore, the function $\Phi_1(w)$ is the Mellin transform of the function $x \mapsto x^{-s/2} \Gamma(s/2,x)$, and this holds for $\re(w)>\max(0,\re(s)/2)$. The function $\Phi_2(w)$ is the Mellin transform of $x^{-s/2}f(x)$, for $|\re(w-s/2)|<a/2$. The reader may check that the half-plane $\re(w)>\max(0,\re(s)/2)$ intersects with the vertical strip $|\re(w-s/2)|<a/2$ as long as $\re(s)>-a$. Thus we can use the Mellin convolution identity and conclude that for $x>0$  and $\re(s)>-a$
$$
h(s,x)=\int_0^{\infty} t^{-s/2-1} \big(x/t)^{-s/2} \Gamma(s/2,x/t) f(t) \d t,
$$
and from this formula we obtain the desired result \eqref{h_s_x_formula} by change of variables $t=x/v$. Thus we have established formula \eqref{h_s_x_formula} for $\re(s)>-a$ and $x>0$, and we can extend it to the sector $|\arg(x)|<|b|$ using the analyticity and asymptotic properties of $f(x)$ that we established above.
\end{proof}

Next we present five applications of the above results. 
Theorem \ref{thm_main} is intended to be used in the following way. We start with a function $g$ satisfying the necessary assumptions and then we try to identify a function $h(s,x)$ via its Mellin transform 
\eqref{h_Melin_transform}. Sometimes we can identify $h(s,x)$ by finding the desired Mellin transform pair in a table of integral transforms, such as \cite{Jeffrey2007} or \cite{Oberhettinger}. This is the method we use in examples 1, 3 and 5 below. When the desired Mellin transform pair can not be found, we may attempt to find an integral or a series representation for $h(s,x)$ -- this is the approach we follow in examples 2 and 4.

\subsection*{Example 1}

Take $g(z)=1/z$. This function satisfies conditions of Theorem \ref{thm_main} with $b=0$ and $a=+\infty$. 
Comparing formula \eqref{incomplete_Gamma_Mellin} with the Mellin transform equation \eqref{h_Melin_transform} that defines $h(s,x)$, we find 
$$
h(s,x)=\frac{1}{2} x^{-s/2} \Gamma(s/2,x), 
$$
and applying Theorem \ref{thm_main} we obtain the following result, valid for 
 for $s\in \c$ and 
$|\textnormal{arg}(\tau)|<\pi/2$:
\begin{align}
\pi^{-s/2} \Gamma(s/2) \zeta(s)&=-\frac{\tau^{s/2}}{s}+\pi^{-s/2}\sum\limits_{n\ge 1} n^{-s}
\Gamma(s/2,\pi \tau n^2) 
\\ \nonumber
&+\frac{\tau^{(s-1)/2}}{s-1} +
\pi^{(s-1)/2}\sum\limits_{n\ge 1} 
n^{s-1}\Gamma((1-s)/2,\pi \tau^{-1} n^2). 
\end{align}
The above formula is well known.  A special case of this formula with $\tau=1$ was used by Riemann to prove the functional equation for $\zeta(s)$ (see \cite{Titchmarsh1987}[Section 2.6]). In the limit as $\tau \to \i$ this formula gives us \eqref{PC_formula1}.  

\vspace{0.25cm}

\subsection*{Example 2}
Now we take $g(z)=e^{\alpha z^2}/z$. Conditions of Theorem \ref{thm_main} are satisfied with $a=+\infty$ and any $b<0$. Berry and Keating 
\cite{Berry_Keating1992} used this function when deriving the asymptotic term $Z_0(t,K)$ in \eqref{BK_formula}, though their method was slightly different. The difference lies in formula \eqref{def:F1}, where we used a function $G(s)=\pi^{-s/2} \Gamma(s/2) \zeta(s)$ and Berry and Keating used $Z(\i (s-1/2))$.  For this choice of $g(z)$ it seems that there are no  explicit representations for $h(s,x)$, however Proposition \ref{prop1} provides a simple integral representation.  Performing a change of variables $x=e^{y}$ and using the Gaussian integral one can easily check that the function $f$ defined by the Mellin transform identity \eqref{def_f0} is given by
$$
f(x)=\frac{1}{8 \sqrt{\pi \alpha}} \exp\Big(-\frac{1}{16 \alpha}\ln(x)^2 \Big), \;\;\; x>0. 
$$
Note that the function $2f(x)/x$ is the probability density function of a random variable $\xi$ having lognormal distribution with parameters $\mu=0$ and $\sigma^2=8\alpha$. Therefore, formula \eqref{h_s_x_formula} gives us the following expression 
$$
h(s,x)=\frac{1}{2} x^{-s/2} {\mathbb E} \big[  \Gamma\big(s/2,x e^{\sqrt{8\alpha} X}\big) \big], 
$$
where $X$ is a standard normal random variable. 
 
%

\vspace{0.25cm}

\subsection*{Example 3}

For our third example we choose 
$$
g(z)=\frac{\pi/2}{\sin(\pi z/2)}.
$$ 
This function satisfies conditions of Theorem \ref{thm_main} with $b=-\pi$ and $a=2$. 
Formula 5.26 from \cite{Oberhettinger} states that 
$$
\int_0^{\infty}  e^x \Gamma(\nu,x) x^{z-1} \d x=\frac{\pi \Gamma(z)}
{\Gamma(1-\nu) \sin(\pi(\nu+z))}, \;\;\; 0< \re(z)<1-\re(\nu). 
$$
From the above formula and the Mellin transform identity \eqref{h_Melin_transform} we find
$$
h(s,x)=\frac{1}{2}\Gamma(1+s/2)  e^x \Gamma(-s/2,x). 
$$
We recall that the normalized incomplete gamma function is  $Q(s,x):=\Gamma(s,x)/\Gamma(s)$. Applying the reflection formula for the gamma function we can write
$$
h(s,x)=-\frac{\pi/2}{\sin(\pi s/2)} e^{x} Q(-s/2,x).
$$
Now Theorem \ref{thm_main} gives us the following result, which seems to be new: 
\begin{corollary} For $-1<\re(s)<2$ and $|\textnormal{arg}(\tau)|<3\pi/2$ 
\begin{align*}
\pi^{-s/2} \Gamma(s/2) \zeta(s)&=
- \frac{\pi \tau^{s/2} }{\sin(\pi s/2)} \Big[ \frac{1}{2} +
\sum\limits_{n\ge 1} e^{\pi \tau n^2} Q(-s/2,\pi \tau n^2) \Big]
\\
&- \frac{\pi \tau^{(s-1)/2} }{\cos(\pi s/2)} \Big[ \frac{1}{2} +
\sum\limits_{n\ge 1} e^{\pi \tau^{-1} n^2} Q((s-1)/2,\pi \tau^{-1} n^2) \Big].
\end{align*}
\end{corollary}

\vspace{0.25cm}

\subsection*{Example 4}

Example 3 can be extended in the following way.  
Let $r=p/q$ be a positive rational number and define
$$
g(z)=\frac{\pi r}{\sin(\pi r z )}.
$$ 
Now conditions of Theorem \ref{thm_main} are satisfied with $b=-2\pi r$ and $a=1/r$. We claim that in this case $h(s,x)$ can be computed as a sum of two infinite power series: 
\begin{align}\label{eqn_h_s_x_4_main}
h(s,x)&=
-\pi r \sum\limits_{n\ge 0} \frac{ (-1)^n}{  \sin(\pi r(2n+s)}\times \frac{x^n}{n!}
\\ \nonumber &+\frac{1 }{2} \pi x^{-s/2} \sum\limits_{m\ge 0} \frac{(-1)^m}{\sin(\pi (s-m/r)/2)}
\times 
\frac{  x^{m/(2r)}}{\Gamma(1-s/2+m/(2r))}.
\end{align}
We will sketch the proof of formula \eqref{eqn_h_s_x_4_main} (the main ideas and details  are very similar to the proof of Corollary 3.16 in 
\cite{KKPW2014}). 
To establish formula \eqref{eqn_h_s_x_4_main}, we start by writing $h(s,x)$ in the form \eqref{def_h_s_x} : 
\begin{equation}\label{eqn_h_s_x_4}
h(s,x)=  \frac{r}{2 \i} \int_{\gamma-\i \infty}^{\gamma+\i \infty} \Phi(w)  \d w, \;\;\; {\textnormal{ where }} \; \Phi(w) = \frac{ \Gamma(w) x^{-w}}{\sin(\pi r (2w-s))}. 
\end{equation}
We take $s$ to be a fixed number in the vertical strip $-1/(2p) < \re(s) < 0$ and $\gamma$ any number in the interval $(0,1/(4p))$. The reader may want to recall the definition of the interval $I_s$ in the proof of Theorem \ref{thm_main}, to see that this choice of $\gamma$ is legitimate. 
The gamma function in the integrand in \eqref{eqn_h_s_x_4} has simple poles at points 
$-n$, $n=0,1,2,\dots$ and the function $\sin(\pi r (2w-s))^{-1}$ has simple poles at points 
$(1/2) (s-m/r)$, $m\in {
\mathbb Z}$. Due to our restriction on $s$, these two sets do not intersect, thus all the poles of the function $\Phi(w)$ are simple. The periodic structure of the poles of $\Phi$ implies that  we can find a sequence of $\{\gamma_k\}_{k\ge 1}$ decreasing to $-\infty$ such that each point $\gamma_k$ satisfies $\gamma_k<\gamma$ and has distance at least $1/(8p)$ from each pole of $\Phi$. 
 
Now we perform the standard operation of shifting the contour of integration in \eqref{eqn_h_s_x_4} from $\gamma + \i\r$ to $\gamma_k + \i \r$:
\begin{align}\label{eqn_h_s_x_4_proof}
 h(s,x)&= \pi r \sum {\textnormal{Res}}\big(\Phi(w), w=-n\big) \\
 \nonumber &+ \pi r 
 \sum {\textnormal{Res}}\Big(\Phi(w), w=\frac{1}{2}\Big(s-\frac{m}{r}\Big)\Big)+
  \frac{r}{2 \i} \int_{\gamma_k-\i \infty}^{\gamma_k+\i \infty} \Phi(w)  \d w,
\end{align}
where the summation in the first sum is over all $n$ such that $\gamma_k \le - n \le  0$ and the summation in the second sum is over all $m$ such that $\gamma_k \le (s-m/r)/2 \le s/2$. The residues in \eqref{eqn_h_s_x_4_proof}  give us the coefficients in the two infinite series in \eqref{eqn_h_s_x_4_main}. 
The integral in the right-hand side of the \eqref{eqn_h_s_x_4_proof} converges to zero as $k\to +\infty$. This is true since the function $x^{-w} \Gamma(w)$ resricted to the vertical line $\re(w)=\gamma_k$ converges to zero uniformly in $\im(w)$ as $k \to +\infty$ and the term $1/\sin(\pi r (2w-s))$ is bounded on line $\re(w)=\gamma_k$. Here our choice of $\gamma_k$ is crucial -- we need to ensure that the line of integration $w\in \gamma_k + \i \r$ stays away from the poles of the integrand. To summarize, in the limit as $k\to +\infty$, formula \eqref{eqn_h_s_x_4_proof} gives us the desired result \eqref{eqn_h_s_x_4_main}. 

The infinite series representation \eqref{eqn_h_s_x_4_main} is in fact valid for almost all irrational $r$. This can be established in the same way as Corollary 3.16 in 
\cite{KKPW2014}. However, for rational $r=p/q$ the convergence of both series in \eqref{eqn_h_s_x_4_main} is obvious and for small $p$ and $q$ these series can be simplified further. We already saw that in the case when $r=1/2$ (our example 3) we could express $h(s,x)$ in terms of a normalized incomplete gamma function. Let us consider now the case with $r=1/4$. By separating the terms with $n$ even from the terms with $n$ odd in the first sum in 
\eqref{eqn_h_s_x_4_main} and applying reflection formula for the gamma function in the second sum we obtain
\begin{align}
h(s,x)&=-\frac{\pi}{4} \frac{\cos(x)}{\sin(\pi s/4)} + 
\frac{\pi}{4} \frac{\sin(x)}{\cos(\pi s/4)}\\ \nonumber
&+\frac{1}{2} \Gamma(s/2) x^{-s/2} 
{}_1 F_2(1;1/2-s/4, 1-s/4 ; -x^2/4).
\end{align}


\vspace{0.25cm}

\subsection*{Example 5}

There is another form of $g(z)$ that can lead to somewhat explicit expressions for $h(s,x)$. We take  
\begin{equation}\label{def_g_5}
g(z)=\frac{A}{z} \times \frac{\prod\limits_{j=1}^n
\Gamma(\alpha_j+z/2) \Gamma(\alpha_j-z/2)}
{\prod\limits_{j=1}^m \Gamma(\beta_j+z/2) \Gamma(\beta_j-z/2)}
\end{equation}
where  the normalizing constant $C$ is given by
$$
A=\prod\limits_{j=1}^n \Gamma(\alpha_j)^{-2} \prod\limits_{j=1}^m \Gamma(\beta_j)^2.
$$
If $n\ge m$ and the constants $\beta_j$ have positive real part and $\alpha_j$ satisfy $\re(\alpha_j)>1/4$, then $g(z)$ given in \eqref{def_g_5} satisfies the assumptions in Theorem \ref{thm_main} with 
$$
a=2 \times \min \{ \re(\alpha_j) \; : \; 1\le j \le n\} \;\; {\textnormal{   and   }}\;\;  b=\pi (m-n). 
$$ 
For function $g$ given by \eqref{def_g_5}, the function $h(s,x)$ defined via \eqref{def_h_s_x} can be expressed in terms of the Meijer $G$-function (see \cite{Mathai})
\begin{equation}
h(s,x)= A \times G_{p,q}^{n+2,n}\Big( 
\begin{matrix}
{\bf a}-s/2 \\ 0, {\bf b} - s/2
\end{matrix} \Big
\vert x 
\Big ),
\end{equation}
where $p=m+n+1$, $q=m+n+2$ and
\begin{equation*}
{\bf a}:=[1-\alpha_1,\dots,1-\alpha_n,\beta_1,\dots,\beta_m,0], \;\;\; 
{\bf b}:=[1,\alpha_1, \dots, \alpha_n, 1-\beta_1,\dots,1-\beta_m].
\end{equation*}
Note that our example 3, where $h(s,x)$ is given in terms of the incomplete gamma function, can be obtained from \eqref{def_g_5} by setting $n=\alpha_1=1$ and $m=0$. 
 It is possible that other choices of parameters may also lead to simple expressions for $h(s,x)$, however we did not pursue this further.

\section*{Acknowledgements}

The research was supported by the Natural Sciences and Engineering Research Council of Canada. The
author would like to thank two anonymous referees for careful reading of the paper and for providing helpful comments and suggestions.


\end{document}